\documentclass[leqno,10pt]{amsart}
\usepackage{amsmath}
\usepackage{amssymb,amscd}
\usepackage{amsthm}
\usepackage{latexsym}
\usepackage[myheadings]{fullpage}
\usepackage{color}  

\newtheorem{theorem}{Theorem}[section]
\newtheorem{lemma}[theorem]{Lemma}
\newtheorem{definition}[theorem]{Definition}
\newtheorem{proposition}[theorem]{Proposition}
\newtheorem{corollary}[theorem]{Corollary}
\newtheorem{remark}[theorem]{Remark}
\newtheorem{example}[theorem]{Example}

\newcommand{\m}{\mathfrak m}
\newcommand{\n}{\mathfrak n}
\newcommand{\p}{\mathfrak p}

\def\opn#1#2{\def#1{\operatorname{#2}}} 
\opn\spec{Spec}
\opn\depth{depth}
\opn\height{ht}
\opn\chara{char}
\opn\gr{gr}
\opn\ord{ord}
\opn\Ap{Ap}
\opn\F{F}
\opn\PF{PF}
\opn\L{L}
\opn\Max{Max}

\title{When is $\m:\m$ an almost Gorenstein ring?}

\author{Marco D'Anna}
\address{Marco D'Anna - Dipartimento di Matematica e Informatica - Universit\`a degli Studi di Catania - Viale Andrea Doria 6 - 95125 Catania - Italy}\email{mdanna@dmi.unict.it}

\author{Francesco Strazzanti}
\address{Francesco Strazzanti - Dipartimento di Matematica e Informatica - Universit\`a degli Studi di Catania - Viale Andrea Doria 6 - 95125 Catania - Italy}
\email{francesco.strazzanti@gmail.com}

\thanks{The first author was supported by the project ``Propriet\`a algebriche locali e globali di anelli associati a curve e  
ipersuperfici'' PTR 2016-18 - Dipartimento di Matematica e Informatica - Universit\`a degli Studi di Catania \\
The second author was supported by INdAM, more precisely he was ``titolare di un Assegno di Ricerca dell'Istituto Nazionale di Alta Matematica''.}


\date{\today}

\begin{document}

\begin{abstract}
Given a one-dimensional Cohen-Macaulay local ring $(R,\m,k)$, we prove that it is almost Gorenstein if and only if $\m$ is a canonical module of the ring $\m:\m$. Then, we generalize this result by introducing the notions of almost canonical ideal and gAGL ring
and by proving that $R$ is gAGL if and only if $\m$ is an almost canonical ideal of $\m:\m$. 
We use this fact to characterize when the ring $\m:\m$ is almost Gorenstein, provided that $R$ has minimal multiplicity. This is a generalization of a result proved by Chau, Goto, Kumashiro, and Matsuoka in the case in which $\m:\m$ is local and its residue field is isomorphic to $k$. 
\end{abstract}


\maketitle

\section*{Introduction}

Gorenstein rings, among Cohen-Macaulay rings, are a very important class from many points of view:
few years after the definition of Gorestein subvarieties, given by Grothendieck, the famous paper of Bass
``On the ubiquity of Gorenstein rings'' \cite{Bass} showed their prominence in many contexts of commutative and homological algebra. In particular, one-dimensional Gorenstein rings play an important role also in multiplicative ideal theory.

It is well-known that a local ring is Gorenstein if and only if it is Cohen Macaulay and has type equal to one. Since it is possible to construct Cohen Macaulay rings with type $n$ for any $n \geq 1$, it is natural to look for intermediate classes between Gorenstein and Cohen Macaulay rings, closed to Gorenstein rings under 
some respect. Many definitions of such classes were proposed, especially in the last two decades, e.g. almost Gorenstein \cite{BF}, $n$-AGL \cite{CGKM}, nearly Gorenstein \cite{HHS} and generalized Gorenstein \cite{GIKT}; see also \cite{HV} for other classes. 
The notion of almost Gorenstein ring has been probably the most studied. It goes back to 1997, when Barucci and Fr\"oberg introduced it in the analytically unramified case, starting by defining the corresponding notion of almost symmetric semigroup, in numerical semigroup theory. Since then, the research on these rings has been greatly developed and their definition has been generalized first in the one-dimensional case \cite{GMP} and later in the higher dimensional case \cite{GTT}. 

There is a very interesting connection between a one-dimensional almost Gorenstein local ring $(R,\mathfrak m)$ and its endomorphism algebra $B=\mathfrak m:\mathfrak m$. In fact it has been proved in \cite{GMP} that $R$ is a one-dimensional almost Gorenstein ring and has minimal multiplicity if and only if $B$ is Gorenstein. We generalize this fact as follows (see Proposition \ref{Corollary of canonical ideal of m:m}):
$R$ is a one-dimensional almost Gorenstein ring if and only if $\mathfrak m$ is a canonical ideal for $B$. This fact was proved also in the non-noetherian context by Barucci in 
\cite{B} with an additional hypothesis.

The $R$-algebra $B$, in the one-dimensional case, is interesting to study in connection with $R$. In fact it is a step toward the blowing up $R^{\mathfrak m}$ of $\mathfrak m$ 
and it coincides with it if $R$ has minimal multiplicity.
So one can expect that $B$ should be in some way better than $R$, even if one can lose the locality. 
We also remark that, in the one-dimensional case, the algebra $B$ or, more generally, the endomorphism 
ring of an ideal has been studied in many situations to characterize properties of $R$ (see, for example,
\cite{SV}). 

In \cite{CGKM} Chau, Goto, Kumashiro, and Matsuoka proposed the notion of $n$-Almost Gorenstein local ring, briefly $n$-AGL, in order to obtain a stratification of one-dimensional Cohen-Macaulay local rings. More precisely, every such ring is $n$-AGL for some $n$ and Gorenstein rings correspond to the 0-AGL rings, whereas a ring is 1-AGL if and only if it is almost Gorenstein ring but not Gorenstein. Therefore, in this perspective 2-AGL rings are the ones closer to be almost Gorenstein and for this reason their properties have been more studied, see, e.g., \cite{GIT}.
In particular, the notion of 2-AGL gives a partial answer to a natural question that arise studying the
endomorphism algebra $B$: can we characterize when it is almost Gorenstein?
Indeed, in \cite[Corollary 5.3]{CGKM}, it has been shown that if $B$ is local with maximal ideal $\n$ and $R$ has minimal multiplicity, then $R$ is 2-AGL if and only if  $R/\m \cong B/\n$ and $B$ is almost Gorenstein but not Gorenstein.\\

In this paper we deal with the following two questions: \\
- If $R$ has minimal multiplicity, when is $B=\m:\m$ almost Gorenstein?\\
- If we drop the hypothesis of minimal multiplicity, can we prove an analogue of Proposition \ref{Corollary of canonical ideal of m:m}? More precisely, Proposition \ref{Corollary of canonical ideal of m:m} shows that $R$ is a one-dimensional almost Gorenstein ring if and only if $\mathfrak m$ is a canonical ideal for $B$; so we look for a notion of almost canonical ideal, such that a ring is almost Gorestein if and only if it is an almost canonical ideal of itself, and for a larger class of rings, such that $R$ belongs to this class if and only if $\m$ is an almost canonical ideal of $B$.

To these aims we introduce the notion of \emph{almost canonical ideal} which generalizes the concept of canonical ideal in the same way almost Gorenstein rings generalize Gorenstein ones.
Then, we propose the definition of \emph{generalized almost Gorenstein local ring}, briefly gAGL ring, that is a class of rings which includes
properly Gorenstein, almost Gorenstein and 2-AGL local rings and we prove that $R$ is gAGL if and only if $\m$ is an almost canonical ideal of $B$. 
From this result it descends that $B$ is almost Gorenstein if and only if $R$ is gAGL, provided that $R$ has minimal multiplicity. We also obtain that $R$ gAGL implies $B$ almost Gorenstein even if $R$ has not minimal multiplicity.

Some of these results and definitions generalize corresponding facts and concepts given in the numerical
semigroups context (\cite{DS}). Let us notice that this generalization is not straightforward, since we have to deal with two problems that do not appear looking at numerical semigroups: the fact that $B$ could not be local anymore
and the non-residually rational case, i.e. $R/\m \not\cong B/\n_i$, where $n_i$ are the maximal 
ideals of $B$.

The structure of the paper is the following. In Section \ref{Section ideal} we define the almost canonical ideals of a one-dimensional ring, we prove that a local ring $R$ is almost Gorenstein if and only if $R$ is an almost canonical ideal of itself, and we give some characterizations that will be useful in the last section. 
In particular, we show that a ring is almost Gorestein if and only if it is an almost canonical ideal of itself
(Corollary \ref{almost canonical and almost Gorenstein}) and we explain why almost canonical ideals generalize canonical ones in the same way almost Gorenstein rings generalize Gorenstein ones (Remark \ref{type}). 

In the following section we introduce the notion of gAGL ring and, after showing that almost Gorenstein and 2-AGL rings are gAGL, we prove that if $R$ has minimal multiplicity and $B$ is local, the ring $R$ is 2-AGL if and only if $R$ is gAGL and the residue fields of $R$ and $B$ are isomorphic, see Proposition \ref{2-AGL and gAGL}. We also show some examples of gAGL rings that are not 2-AGL, both when $B$ is local and not. 

In the last section we describe the canonical ideal of $B$ (Proposition \ref{Canonical ideal of m:m}) and we use it to prove that $R$ is a one-dimensional almost Gorenstein ring if and only if $\mathfrak m$ is a canonical ideal for $B$ and to give a simpler proof of \cite[Theorem 5.1]{GMP} (Proposition \ref{Corollary of canonical ideal of m:m}). Successively, we prove the main result of the paper (Theorem \ref{main}): $R$ is gAGL if and only if $\m$ is an almost canonical ideal of $B$. Finally, in Corollary \ref{final corollary}, we specialize this result to characterize the almost Gorensteinness of $B$.

\section{Almost canonical ideals of a one-dimensional ring} \label{Section ideal}

Throughout the paper $R$ will be a one-dimensional Cohen-Macaulay ring and, if $I$ and $J$ are two fractional ideals of $R$, we set $I:J=\{r \in Q(R) \mid rJ \subseteq I\}$, where $Q(R)$ denotes the total ring of fractions of $R$; we will always deal with fractional ideals containing an invertible of $Q(R)$, so we will not explicitly mention this hypothesis.
Our first goal is to introduce the notion of almost canonical ideal in this setting, but first it is convenient to recall the corresponding notion for numerical semigroups.

A numerical semigroup $S$ is an additive submonoid of $\mathbb{N}$ such that $\mathbb{N} \setminus S$ is finite. Hence, there exists $\F(S)=\max(\mathbb{N}\setminus S)$ which is called Frobenius number of $S$.  
A set $E\subseteq \mathbb{Z}$ is said to be a relative ideal of $S$ if there exists $s \in S$ such that $s+E \subseteq S$ and $E+S \subseteq E$. It is easy to see that also $\mathbb{Z}\setminus E$ has a maximum and we denote it by $\F(E)$.
For example $M=S \setminus \{0\}$ and $K(S)=\{z \in \mathbb{N}\mid \F(S)-z \notin S\}$ are relative ideals of $S$, called the maximal and the standard canonical ideal of $S$ respectively.
Given two relative ideals $E$, $F$ of $S$, we set $E-F=\{z \in \mathbb{Z} \mid z+F \subseteq E\}$.
Moreover, we also set $\widetilde{E}=E+\F(S)-\F(E)$ to be the unique relative ideal isomorphic to $E$ such that $\F(\widetilde{E})=\F(S)$.

In \cite{DS} a relative ideal $E$ of $S$ is said to be almost canonical if $\widetilde{E}-M=K(S) \cup \{\F(S)\}$. This notion was introduced in order to generalize the concept of almost symmetric numerical semigroup, since $S$ is almost symmetric if and only if $S-M=K(S) \cup \{\F(S)\}$. In \cite[Proposition 2.4]{DS} some equivalent definitions are shown, in particular it is proved that $E$ is almost canonical if and only if $K(S)-(M-M) \subseteq \widetilde{E}$.

To generalize this definition we note that the properties of $S$ are strictly related to those of the one-dimensional local domain $k[[S]]=k[[t^s \mid s \in S]]$, where $k$ is a field and $t$ is an indeterminate. For instance, a fractional ideal $I$ of $k[[S]]$ corresponds to the relative ideal $v(I)$ of $S$, by setting $v(I)$ to be the set of the orders of the elements in $I$. Given a canonical module $\omega_{k[[S]]}$ of $k[[S]]$ such that $k[[S]] \subseteq \omega_{k[[S]]} \subseteq \overline{k[[S]]}=k[[t]]$, it follows that $v(\omega_{k[[S]]})=K(S)$ by \cite[Satz 5]{J}. Moreover, if $I \subseteq k[[S]]$ is an ideal and $(x)$ is a minimal reduction of $I$, then $v(x)=\min (v(I))$.  
Finally, we recall that if $I$ and $J$ are monomial ideals of $k[[S]]$, then $v(I:J)=v(I)-v(J)$.

Now come back to our ring $R$ and assume for a while that $(R,\m,k)$ is also local with $k$ infinite and that there exists a canonical module $\omega_R$ of $R$ with $R \subseteq \omega_R \subseteq \overline{R}$. 
If $I$ is a fractional ideal of $R$, there exist a regular element $y$ of $R$ such that $J=y(\omega_R:I) \subseteq R$ and a minimal reduction $(x)$ of $J$. We say that $z=x/y\in (\omega_R:I)$ is a reduction of $\omega_R:I$ and we fix this element throughout this section. We also notice that there exists a regular $r \in R$ such that $r I \subseteq R \subseteq \omega_R$, so $\omega_R:I$ is a regular ideal and, then, $x$ and $z$ are always non-zero divisors in $Q(R)$. 

If $R=k[[S]]$ and $I$ is a monomial ideal, then $v(z)=v(x)-v(y)=v(y)+\min(K(S)-v(I))-v(y)=\F(S)-\max(\mathbb{Z} \setminus v(I))=\F(S)-\F(v(I))$, where the penultimate equality follows by \cite[Hilfssatz 5]{J}. This means that $v(z)+v(I)=\widetilde{v(I)}$ and, thus, we can generalize the notion of almost canonical ideal using $zI$ in place of $\widetilde{v(I)}$. We denote the fractional ideal $\omega_R:I$ by $I^{\vee}$.

\begin{definition} \rm
If $(R,\m)$ is a Cohen Macaulay local ring, with infinite residue field and with a canonical module $\omega_R$ such that $R \subseteq \omega_R \subseteq \overline{R}$, we say that $I$ is almost canonical if $\omega_R:(\m:\m) \subseteq zI$ or equivalently $z^{-1}I^{\vee} \subseteq (\m:\m)$. If $R$ is not local, we say that $I$ is almost canonical if $I_{\m}$ is an almost canonical ideal of $R_{\m}$ for every maximal ideal $\m$ of $R$.  
\end{definition}

Clearly if $R=k[[S]]$ and $I$ is a monomial ideal of $R$, then $I$ is an almost canonical ideal of $R$ if and only if $v(I)$ is an almost canonical ideal of $S$.  
We point out that in the definition above we are assuming that for every maximal ideal $\m$ of $R$ there exists a canonical module $\omega_{R_\m}$ of $R_\m$ such that $R_\m \subseteq \omega_{R_\m} \subseteq \overline{R_\m}$. 

\begin{remark} \rm \label{First remark}
{\bf 1.} If $R$ is local and $\omega'_{R}$ is another canonical module of $R$ included between $R$ and $\overline{R}$, then $\omega'_{R}=r\omega_{R}$ for some invertible element $r \in Q(R)$. Therefore, it is easy to see that the definition above does not depend on the chosen canonical module. 
Moreover, it is also independent of the choice of $z$. In fact, if $y$ and $y'$ are two regular elements such that 
$yI^{\vee}, y'I^{\vee} \subseteq R$ and $x$ is a minimal reduction of $yI^{\vee}$, then $xy'/y$ is a minimal
reduction of $y'I^{\vee}$ and $z=x/y=(xy'/y)(1/y')$; hence, it suffices to check that if we choose two minimal
reductions $x,x'$ of $yI^{\vee}$, then $z^{-1}I^{\vee} \subseteq (\m:\m)$ if and only if 
$z'^{-1}I^{\vee} \subseteq (\m:\m)$,
with $z'=x'/y$. This equivalence is verified, because the blowing-up of $yI^{\vee}$ is $R[yI^{\vee}/x]=
R[yI^{\vee}/x']$, i.e. 
$R[z^{-1}I^{\vee}]=R[z'^{-1}I^{\vee}]$, and $\m:\m$ is a ring. \\ 
{\bf 2.} Every canonical ideal is almost canonical. Indeed, if $I=r \omega_R$ for some invertible element $r\in Q(R)$, then it is straightforward to see that we can choose $z=r^{-1}$ and $\omega_R:(\m:\m) \subseteq \omega_R$ is equivalent to $R \subseteq (\m:\m)$. \\ 
{\bf 3.} If $I$ and $J$ are two isomorphic fractional ideals of $R$, then $I$ is almost canonical if and only if $J$ is almost canonical. To show this we note that $J=(r/s)I$ for some regular elements $r,s \in R$ and that we can easily reduce to assume that $R$ is local. Let $y \in R$ be a regular element such that $y(\omega_R:I)\subseteq R$ and let $x$ be a minimal reduction of $y(\omega_R:I)$. Then, $ry(\omega_R:J)=sy(\omega_R:I)\subseteq R$ and $sx$ is a minimal reduction of $ry(\omega_R:J)$. Hence, $\omega_R:(\m:\m) \subseteq (x/y)I$ if and only if $\omega_R:(\m:\m) \subseteq (sx/ry)J=(x/y)I$. \\ 
{\bf 4.} Since $z \in (\omega_R:I)$, it always holds that $zI \subseteq \omega_R$. Therefore, almost canonical ideals are the ideals satisfying $\omega_R:(\m:\m) \subseteq zI \subseteq \omega_R$, i.e. $R \subseteq z^{-1}I^{\vee} \subseteq (\m:\m)$. \\
{\bf 5.} In \cite{BDS2} almost canonical ideals naturally appear characterizing the almost Gorenstein property of some quadratic quotients of the Rees algebra $R[It]$ of $R$ with respect to a proper ideal $I$. More precisely, if $a$, $b \in R$, let $J$ denote the contraction of the ideal $(t^2+at+b)R[t]$ to $R[It]$ and let $R(I)_{a,b}=R[It]/J$; see \cite{BDS,BDS2, OST} for the importance of this family of rings. If $R$ is a one-dimensional Cohen-Macaulay local ring, \cite[Corollary 2.4]{BDS2} says that $R(I)_{a,b}$ is an almost Gorenstein ring if and only if $I$ is an almost canonical ideal of $R$ and $z^{-1}I^{\vee}$ is a ring. 
\end{remark}

\begin{lemma} \label{z=1}
Let $I \subseteq \overline{R}$ be a fractional ideal of $R$ containing $1$. Then, a reduction of the ideal $I_{\m}$ of $R_{\m}$ is $1$ for every maximal ideal $\m$ of $R$. 
\end{lemma}

\begin{proof}
It is enough to assume that $R$ is local. Since $R$ is noetherian, $R[I]$ is a finitely generated $R$-algebra. Moreover, it is finite over $R$, because $R \subseteq R[I] \subseteq \overline{R}$. This implies that $R[I]$ is a noetherian $R$-module and, thus, the chain $I \subseteq I^2 \subseteq I^3 \subseteq \dots$ stabilizes, i.e. $I^n=I^{n+1}$ for some $n \in \mathbb{N}$. If $y \in R$ is such that $yI \in R$, it follows $(yI)^{n+1}=y^{n+1}I^{n+1}=yy^nI^n=y(yI)^{n}$ and, so, $y/y=1$ is a reduction of $I$.  
\end{proof}

A one-dimensional Cohen-Macaulay local ring $R$ is said to be almost Gorenstein if $\m \omega_R \subseteq \m$, where $\omega_R$ is a canonical module included between $R$ and $\overline{R}$. If $R$ is not local, we say that $R$ is almost Gorenstein if $R_\m$ is almost Gorenstein for every maximal ideal $\m$ of $R$. 
When $I=R$, the previous lemma easily implies that $z=1$ is a reduction of $\omega_{R_\m}:I_{\m}=\omega_{R_\m}$ for every maximal ideal $\m$ of $R$. Hence, $\omega_{R_\m}:(\m R_\m:\m R_\m) \subseteq R_{\m}$ if and only if $\omega_{R_\m} \subseteq (\m R_\m:\m R_\m)$. Hence, we have proved the following:

\begin{corollary} \label{almost canonical and almost Gorenstein}
 A ring $R$ is almost Gorenstein if and only if it is almost canonical as ideal of itself. 
\end{corollary}

\begin{remark} \rm \label{containment between fractional ideals}
Let $I \subseteq J$ be two fractional ideals of $R$ such that $R \subseteq (\omega_R:J) \subseteq (\omega_R:I) \subseteq \overline{R}$ and assume that $I$ is almost canonical. Lemma \ref{z=1} ensures that $1$ is a reduction of both $(\omega_R:I)_{\m}$ and $(\omega_R:J)_{\m}$ for every maximal ideal $\m$ of $R$ and, then, the definition immediately implies that also $J$ is almost canonical. 
\end{remark}

Now we give two useful criteria to prove that an ideal is almost canonical.

\begin{proposition} Let $(R,\m)$ be a local ring. The following conditions are equivalent:
\begin{enumerate}
\item The ideal $I$ is almost canonical; 
\item $(\omega_R:\m)=(zI:\m)$;
\item $\m\omega_R \subseteq zI$.
\end{enumerate}
\end{proposition}

\begin{proof}
(1) $\Rightarrow$ (2) The inclusion $(zI:\m) \subseteq (\omega_R:\m)$ is always true because $z \in (\omega_R:I)$. Moreover, if $\lambda \in (\omega_R:\m)$, $m \in \m$ and $n \in (\m:\m)$, then $\lambda m n \in \omega_R$, that is $\lambda m \in (\omega_R:(\m:\m))\subseteq zI$ and, hence, $\lambda \in (zI:\m)$. \\
(2) $\Rightarrow$ (1) Let $\alpha \in z^{-1}I^{\vee}$ and let $m \in \m$. By hypothesis $\omega_R : (zI:\m)=\omega_R : (\omega_R : \m)=\m$. If $\beta \in (zI:\m)$, then $\alpha m \beta \in \alpha zI \subseteq \omega_R$ and, thus, $\alpha m \in (\omega_R: (zI:\m))=\m$. Hence, $\alpha \in (\m:\m)$ and $I$ is an almost canonical ideal. \\
(1) $\Leftrightarrow$ (3)  If $R$ is a DVR with uniformizing parameter $t$, then every non-zero fractional ideal of $R$ is generated by $t^n$ for some $n \in \mathbb{Z}$. It is straightforward to check that in this case we can choose $z=t^{-n}$. Therefore, $\m R \subseteq t^{-n} (t^n)=R$ is always true. Moreover, every fractional ideal of $R$ is canonical and, thus, almost canonical. If $R$ is not a DVR, then $R:\m=\m:\m$ and 
\[
R:\m=\m:\m \Longleftrightarrow (\omega_R:\omega_R):\m=\m:\m \Longleftrightarrow \omega_R:\m\omega_R =\m:\m \Longleftrightarrow \m\omega_R = \omega_R:(\m:\m).
\]
Hence, $I$ is an almost canonical ideal of $R$ if and only if $\m\omega_R \subseteq zI$. 
\end{proof}

\begin{remark}\label{type}\rm
If $R$ is local, we always have that $zI \subseteq \omega_R$ and, then, $(zI:\m) \subseteq (\omega_R:\m)$. Consequently, the type of $I$ is equal to
\begin{align*}
t(I)&=\ell_R\left(\frac{I:\m}{I}\right)=\ell_R\left(\frac{zI:\m}{zI}\right)\leq \ell_R\left(\frac{\omega_R:\m}{zI}\right)= \\
&=  \ell_R\left(\frac{\omega_R}{zI}\right) + \ell_R\left(\frac{\omega_R:\m}{\omega_R}\right) =  \ell_R\left(\frac{\omega_R}{zI}\right) + \ell_R\left(\frac{R}{\m}\right)=\ell_R\left(\frac{\omega_R}{zI}\right)+1.
\end{align*}
We notice that all the lengths above are finite because all the modules are fractional ideals containing a regular element and $R$ is one-dimensional.
Moreover, the previous proposition implies that $t(I)=\ell_R(\omega_R/zI)+1$ if and only if $I$ is an almost canonical ideal. Also, when $I=R$ we recover the well-known fact that $t(R)=\ell_R(\omega_R/R)+1$ if and only if $R$ is almost Gorenstein, see \cite[Definition-Proposition 20]{BF} for the analytically unramified case.
\end{remark}

We denote the set of maximal ideals of $R$ by $\Max(R)$ and the Jacobson radical of $R$ by $J(R)$.

\begin{corollary} \label{Characterization almost canonical ideal}
Let $R$ be a semilocal ring, $I$ be a fractional ideal of $R$ and assume that it is possible to choose a reduction $z=1$ for the fractional ideal $(\omega_R:I)_\m$ of $R_\m$ for every $\m \in \Max(R)$. The following conditions are equivalent:
\begin{enumerate}
\item $I$ is an almost canonical ideal;
\item $(\omega_R:J(R))=(I:J(R))$;
\item $J(R) \omega_R \subseteq I$.
\end{enumerate}
\end{corollary}

\begin{proof}
Since $R$ is semilocal, $J(R)_{\m}=J(R_{\m})=\m R_{\m}$ for every $\m \in \Max(R)$, since localization commutes with finite intersections. Given $\m \in \Max(R)$, it follows from the previous proposition and the hypothesis that $I_{\m}$ is almost canonical if and only if $(\omega_{R_{\m}}:\m R_{\m})=(I_{\m}:\m R_{\m})$, that is equivalent to $(\omega_{R}:J(R))_{\m}=(I:J(R))_\m$. Since the equality is a local property we get (1) $\Leftrightarrow$ (2). The equivalence (1) $\Leftrightarrow$ (3) is similar.
\end{proof}

\section{Generalized almost Gorenstein local rings}

Let $(R,\m,k)$ be a one-dimensional Cohen-Macaulay local ring with a canonical module $\omega_R$ such that $R \subseteq \omega_R \subseteq \overline{R}$. Following \cite[Lemma 3.1]{CGKM}, we say that $R$ is 2-AGL if $\omega_R^2=\omega_R^3$ and $\ell_R(\omega_R^2/\omega_R)=2$. 
We also notice that a ring is almost Gorenstein if and only if the length $\ell_R(\omega_R^2/\omega_R)$ is less than 2, being zero only in the Gorenstein case; see \cite[Theorem 3.16]{GMP}.

In this section we introduce a new class of rings which includes both almost Gorenstein and 2-AGL rings; moreover, this will be a generalization of the notion of GAS numerical semigroup given in \cite{DS}.
First of all we notice that in a GAS numerical semigroup $S$ we always have $2M \subseteq (S-K(S))$, see \cite[Proposition 3.5]{DS}, so we assume that $\m^2 \subseteq (R:\omega_R)$. Let $\{x_1, \dots, x_n\}$ be a minimal system of generators of $\m$. If $R$ is not Gorenstein, we have $(R:\omega_R) \subseteq \m$, then $n=\ell_R(\m/\m^2)\geq \ell_R(\m/(R:\omega_R))=\ell_k(\m/(R:\omega_R))=:r$ and it is easy to see that, up to reorder the generators, $R:\omega_R=(x_1, \dots, x_r)^2+ (x_{r+1},\dots, x_n)$.
Clearly $r=0$ if and only if $\m \subseteq (R:\omega_R)$, i.e. $R$ is almost Gorenstein (see the remark below). Moreover, in \cite[Proposition 3.3]{CGKM} it is proved that if $R$ is 2-AGL, then $r=1$.

\begin{remark} \rm \label{R:omega}
{\bf 1.} If $R$ is not Gorenstein, we claim that $R:\omega_R=\m:\omega_R$.
If $x \in (R:\omega_R)\setminus (\m :\omega_R)$, then $x\omega_R$ is an $R$-module contained in $R$ but not in $\m$ and, thus, $x \omega_R =R$. Moreover, $x \in R$ because $1 \in \omega_R$, while $1 \in x \omega_R$ implies that $x$ is a regular element for $\omega_R$. Hence, $R \cong \omega_R$ and $R$ is Gorenstein. \\
{\bf 2.} If $R$ is not Gorenstein, it is straightforward to see that $R:\omega_R$ is also an $(\m:\m)$-module. In fact if $a \in (\m:\m)$, $b \in (R:\omega_R)=(\m:\omega_R)$ and $c \in \omega_R$, it follows that $abc \in (\m:\m)\m \subseteq \m \subseteq R$. 
\end{remark}

The next lemma will give raise to the definition of gAGL ring. Set $B=\m:\m$ and recall that $J(B)$ denotes the Jacobson radical of $B$.

\begin{lemma} \label{Def gAGL}
Assume that $\m^2 \subseteq (R:\omega_R)$ and let $I$ be an ideal of $B$ contained in $J(B)$. The following conditions are equivalent:
\begin{enumerate}
\item For every minimal system of generators $\{x_1, \dots, x_n\}$ of $\m$ for which $R:\omega_R=(x_1, \dots, x_r)^2 + (x_{r+1}, \dots, x_n)$ there are no $x \in I$ such that $x_j=x x_i$ for some $i,j=1, \dots, r$;
\item $I \subseteq (R:\omega_R):\m$.
\end{enumerate}
\end{lemma}

\begin{proof} 
(1) $\Rightarrow$ (2) We can assume that $R$ is not Gorenstein, otherwise $(R:\omega_R):\m=R:\m \supseteq B \supseteq I$.
Let $\{x_1, \dots, x_n\}$ be a minimal system of generators of $\m$ and assume that $R:\omega_R=(x_1, \dots, x_r)^2+(x_{r+1}, \dots, x_n)$ with $r \geq 0$. Given $x \in I$, we need to prove that $x x_i \in (R:\omega_R)$ for every $i=1, \dots, n$. If $i>r$, then $x x_i$ is in $R:\omega_R$ because this is a $B$-module. Let $1 \leq i \leq r$ and assume by contradiction that $xx_i \in \m \setminus (R:\omega_R)$. Thus, $xx_i=\sum_{j=1}^r \lambda_j x_j + \sum_{j=r+1}^n \mu_j x_j + y$ where $\lambda_j$ and $\mu_j$ are zero of units of $R$ and $y\in \m^2$. Since $xx_i \notin (R:\omega_R)$, there is at least one $\lambda_j$ different from zero. If $\lambda_j \neq 0$ for some $j \neq i$, then we set $\overline{x_j}=xx_i$. The images of the elements in $X=\{x_1, \dots, x_{j-1},\overline{x_j},x_{j+1}, \dots, x_n\}$ are a basis of the $R/\m$-vector space $\m/\m^2$ and, then, $X$ is a minimal system of generators of $\m$. This is a contradiction because $\overline{x_j}=xx_i$ and $x \in I$. On the other hand, if $\lambda_j=0$ for all $j \neq i$, then $\lambda_i \neq 0$ and $xx_i=\lambda_i x_i+\sum_{j={r+1}}^n \mu_j x_j+y=\lambda_i x_i + \delta$ with $\delta \in (R:\omega_R)$. Since $x \in I \subseteq J(B)$ and $\lambda_i$ is a unit of $R$, it follows that $x-\lambda_i$ is a unit of $B$ and, thus, $x_i=\delta (x-\lambda_i)^{-1} \in (R:\omega_R)$ because $R:\omega_R$ is a $B$-module. This yields a contradiction and, then, we get $I \subseteq (R:\omega_R):\m$. \\
(2) $\Rightarrow$ (1) If $x_j=x x_i$ with $x \in I\subseteq (R:\omega_R):\m$ for some $1 \leq i, j \leq r$, then $x_j=x x_i \in (R:\omega_R)$ yields a contradiction.
\end{proof}

Even though the very definition of GAS numerical semigroup is related to the pseudo-Frobenius numbers of the semigroup, that are in turn connected to the generators of the socle of the associated ring, in \cite[Proposition 3.5]{DS} it is proved that a numerical semigroup $S$ is GAS if and only if $x-y$ is not in the maximal ideal of $M-M$ for every $x,y \in M \setminus (S-K(S))$. Moreover, excluding the Gorenstein case, it is showed that this is also equivalent to say that $2M \subseteq S-K(S) \subseteq M$ and the maximal ideal of $M-M$ is contained in $(S-K(S))-M$. Therefore, also in light of the previous lemma, the following definition arises naturally.

\begin{definition}
Let $(R,\m)$ be a one-dimensional Cohen-Macaulay local ring with a canonical module $\omega_R$ such that $R \subseteq \omega_R \subseteq \overline{R}$. We say that $R$ is a generalized almost Gorenstein local ring, briefly {\rm gAGL}, if $J(B) \subseteq (R:\omega_R):\m$, i.e. $\m \omega_R J(B) \subseteq R$.
\end{definition}

A priori the definition above depends on $\omega_R$, but we will show later that it is independent of the chosen canonical module, see Remark \ref{independence}. Therefore, it is straightforward to see that $k[[S]]$ is gAGL if and only if $S$ is GAS: simply choose $\omega_{k[[S]]}=(t^s \mid s \in K(S))$.

\begin{remark} \rm \label{Almost Gorenstein are gAGL}
{\bf 1.} The containment $\m:\m \subsetneq (R:\omega_R):\m$ holds if and only if $R$ is a DVR. In fact, if it is strict, then $\m \subsetneq (R:\omega_R) \subseteq R$ because $1 \in \omega_R$ and, thus, $R:\omega_R=R$. Moreover, $\m:\m \subsetneq R:\m$ if and only if $R$ is a DVR. The converse can be proved in the same way, since $R=\omega_R$. Hence, a DVR is gAGL. \\
{\bf 2.} If $R$ is not a DVR, then $(R:\omega_R):\m=\m:\m$ if and only if $R$ is almost Gorenstein. Indeed if $R$ is Gorenstein, then $(R:\omega_R):\m=R:\m=\m:\m$. If $R$ is not Gorenstein, then $1 \notin (R:\omega_R)$ and, so, $(R:\omega_R) \subseteq \m$ because $1 \in \omega_R$. This means that $(R:\omega_R):\m \subseteq \m:\m$ and, since $(R:\omega_R):\m$ is a $(\m:\m)$-module, the equality holds if and only if $1 \in (R:\omega_R):\m=(R:\m):\omega_R=(\m:\m):\omega_R$ or equivalently $\omega_R \subseteq \m:\m$, i.e. $R$ is almost Gorenstein. In particular, an almost Gorenstein ring is gAGL. \\
{\bf 3.} If $R$ is 2-AGL, then $R:\omega_R=(x_1)^2+(x_2, \dots, x_n)$ for every minimal system of generators by \cite[Proposition 3.3]{CGKM} and, therefore, $R$ is gAGL by Lemma \ref{Def gAGL}. \\
{\bf 4.} Let $k$ be a field and let $t$ be an indeterminate. The ring $R=k[[t^6,t^7,t^9,t^{17}]]$ is neither almost Gorenstein nor 2-AGL, but it is gAGL because the associated numerical semigroup is GAS. Indeed $\omega_R$ is generated by $1$, $t$, $t^3$ and $R:\omega_R=(t^7,t^9)^2 + (t^6,t^{17})$.
\end{remark} 

We are going to show that if $R$ has minimal multiplicity and $(B,\n)$ is local, then $R$ is 2-AGL if and only if it is gAGL, but not almost Gorenstein, and $R/\m \cong B/\n$. 
We first recall that $\alpha \m=\m^2$ if and only if $B=\alpha^{-1}\m$, where $\alpha \in \m$. It is well-known that this condition is equivalent to say that $R$ has minimal multiplicity, i.e. $\nu(R)=e(R)$, where $\nu(R)$ denotes the embedding dimension of $R$ and $e(R)$ its multiplicity. We start with a lemma.

\begin{lemma} \label{R:omega^2} Let $\alpha \in \m$. The following properties hold:
\begin{enumerate}
\item If $R$ is {\rm gAGL}, $(B,\n)$ is local and $R/\m \cong B/\n$, then $\m^2 \subseteq R:\omega_R^2$;
\item If $R$ is not Gorenstein and $\alpha \m =\m^2$, then $(R:\omega_R):\m=\{m/\alpha \mid m \in (R:\omega_R) \}$;
\item If $\alpha \m=\m^2$ and $R$ is not almost Gorenstein, then $\alpha \notin (R:\omega_R)$;
\item There exists a minimal system of generators of $\m$ whose elements are regular.
\end{enumerate}
\end{lemma}

\begin{proof}
(1) We can assume that $R$ is not a DVR otherwise $R:\omega_R^2=R$. If $\m = \n$ it follows
\[
1=\ell_{B/\n}(B/\n)=\ell_{R/\m}(B/\n)=\ell_{R/\m}(B/\m)=\ell_{R/\m}(B/R)+\ell_{R/\m}(R/\m) = \ell_{R/\m}(B/R)+ 1
\]
and then $\ell_{R/\m}(B/R)=0$, i.e. $B=R$, that is equivalent to say that $R$ is a DVR. Thus, assume $\m \neq \n$.
Since $R$ is gAGL, $\n \m \subseteq (R:\omega_R)$ or equivalently $\m \omega_R \subseteq (R:\n)$. Moreover, $R:\n \subseteq R:\m=\m:\m$ and in $R:\n$ there are no units $\lambda$ of $B$ otherwise $\lambda \n \subseteq R$ implies $\n=\lambda \n \subseteq \m$. Hence, we get $\m \omega_R \subseteq \n \subseteq (R:\omega_R):\m$ that implies $\m^2 \omega_R \subseteq R:\omega_R$ and, then, $\m^2 \subseteq R:\omega_R^2$.\\
(2) Let $x$ be an element of $(R:\omega_R):\m$. Since $(R:\omega_R) \subseteq \m$, then $x \in (\m:\m)=\alpha^{-1}\m$ and, so, $x= \alpha^{-1}m$ for some $m \in \m$. Moreover, $m =x \alpha \in (R:\omega_R)$. Conversely, if $m \in (R:\omega_R)=(\m:\omega_R)$ and $\lambda \in \omega_R$, then $\alpha^{-1}m \lambda \in \alpha^{-1}\m=\m:\m=R:\m$ and $\alpha^{-1}m \in (R:\m):\omega_R=(R:\omega_R):\m$. \\
(3) If $\alpha \in (R: \omega_R)=\m : \omega_R$, then $\alpha \omega_R \subseteq \m$ implies that $\omega_R \subseteq \alpha^{-1} \m=\m:\m$. Therefore, $\omega_R \m \subseteq \m$ and, then, $R$ is almost Gorenstein. \\
(4) Let $\p_1, \dots, \p_s$ be the associated prime ideals of $R$ and suppose that there is an element $x$ in a minimal system of generators of $\m$ that is not regular. We assume that $x \in (\p_1 \cap \dots \cap \p_t)$ and $x \notin (\p_{t+1} \cup \dots \cup \p_s)$ with $t \geq 1$. If $t=s$, we set $y=1$, otherwise we choose $y \in (\p_{t+1} \cap \dots \cap \p_s)\setminus (\p_1 \cup \dots \cup \p_t)$ that exists by the prime avoidance lemma and because $R$ is Cohen-Macaulay. Pick also $z \in \m^2 \setminus (\p_1 \cup \dots \cup \p_s)$. Then, $x+yz$ is a non-zero divisor because $x\in (\p_{1} \cap\dots \cap \p_{t})\setminus (\p_{t+1} \cup \dots \cup \p_s)$ and $yz \in (\p_{t+1} \cap\dots \cap \p_{s})\setminus (\p_{1} \cup \dots \cup \p_t)$. Moreover, $x$ and $x+yz$ have the same image in $\m/\m^2$ and, thus, it is possible to replace $x$ by $x+yz$.
\end{proof}

\begin{proposition} \label{2-AGL and gAGL}
Assume that $(B,\n)$ is local and $\alpha \m=\m^2$ for some $\alpha \in \m$. Then, $R$ is {\rm 2-AGL} if and only if $R/\m \cong B/\n$ and $R$ is {\rm gAGL} but not almost Gorenstein.
\end{proposition}

\begin{proof}
If $R$ is 2-AGL, then $R/\m \cong B/\n$ by \cite[Theorem 5.2 (2)]{CGKM} and $R$ is gAGL by Remark \ref{Almost Gorenstein are gAGL}.

Assume now that $R/\m \cong B/\n$ and $R$ is gAGL but not almost Gorenstein. By Lemma \ref{R:omega^2}(3) it follows that $\alpha \notin (R: \omega_R)$. Clearly $\alpha \notin \m^2$, then we can choose a minimal system of generators of $\m$ containing $\alpha$, say $\alpha=x_1, x_2, \dots, x_n$. Since $R$ is gAGL and $\alpha \notin (R: \omega_R)$, we can assume that $R: \omega_R=(\alpha, x_2, \dots, x_r)^2+(x_{r+1}, \dots, x_n)$ for some $r \geq 1$. If $r \neq 1$, then $\alpha^{-1} x_2 \in \alpha^{-1}\m\setminus \n=B \setminus \n$ by Lemma \ref{Def gAGL}. Since $B/\n \cong R/\m$, there exists $\lambda \in R \setminus \m$ such that $\alpha^{-1} x_2 - \lambda \in \n \subseteq (R:\omega_R):\m$, that implies $x_2 - \lambda \alpha \in (R: \omega_R)$. Therefore, if we consider the images in $\m/\m^2$, we get $\overline{x_2} - \overline{\lambda} \overline{\alpha}= \sum_{i={r+1}}^n \overline{\lambda_i} \overline{x_i}$ with $\overline{\lambda}, \overline{\lambda_{r+1}}, \dots, \overline{\lambda_{n}} \in R/\m$, but this is a contradiction because $\alpha, x_2, x_{r+1}, \dots, x_n$ are linearly independent over $R/\m$. Hence, we get $r=1$. 

By \cite[Proposition 2.3(2) and Theorem 3.7(7)]{CGKM}, it is enough to prove that $R[\omega_R]=\omega_R^2$. This holds if and only if $\omega_R^2=\omega_R^3$, that is in turn equivalent to $\omega_R:\omega_R^2=\omega_R:\omega_R^3$, i.e. $R:\omega_R=R:\omega_R^2$. One inclusion is always true, so assume by contradiction that $(R:\omega_R^2) \subsetneq (R:\omega_R)$. 
By Lemma \ref{R:omega^2}(4) we can choose a minimal system of generators $X$ of $\m$ whose elements are regular and, since $\m^2 \subseteq (R:\omega_R^2)$ by Lemma \ref{R:omega^2}(1), there exists $x \in X$ such that $x \in (R:\omega_R) \setminus (R:\omega_R^2)$. Thus, there are $w_1$, $w_2 \in \omega_R$ such that $xw_1 w_2 \notin R$. In particular, $xw_1 \notin (R:\omega_R)$ and, since $x\in (\m:\omega_R)$ by Remark \ref{R:omega}.1, $xw_1 \in \m$. It follows that $xw_1=\lambda \alpha+y$, where $\lambda$ is a unit of $R$ and $y\in (R:\omega_R)$. In $Q(R)$ we have the equality $\alpha^{-1}xw_1=\lambda + \alpha^{-1}y$ and, since $R$ is gAGL, $\alpha^{-1}y \in (R:\omega_R):\m = \n$ by Lemma \ref{R:omega^2}(2) and Remark \ref{Almost Gorenstein are gAGL}.2. 
Therefore, $\lambda + \alpha^{-1}y$ is a unit of $B$ and of $\overline{R}$. Hence, $\alpha x^{-1}=w_1(\lambda + \alpha^{-1}y)^{-1} \in \overline{R}$.
Since $\alpha^{-1} x\in \alpha^{-1}\m=B \subseteq \overline{R}$, it follows immediately that $(x)\overline{R}=(\alpha) \overline{R}$. Also, since $(\alpha)$ is a reduction of $\m$, we get $(x) \subseteq \m \subseteq \overline{(\alpha)} =(\alpha) \overline{R} \cap R= (x) \overline{R} \cap R= \overline{(x)}$ and this implies that $(x)$ is a reduction of $\m$, see \cite[Corollary 1.2.5]{SH}. Moreover, we have that $x\m=\m^2$ by \cite[Corollary 2.2]{M}. Hence, $x \notin (R:\omega_R)$ by Lemma \ref{R:omega^2}(3) and this yields a contradiction.
\end{proof} 

\begin{example} \label{Example 2} \rm
Let $k$ be a field and let $R=k+(t^3k[[t^3,t^4,t^5]]\times u^3k[[u^3,u^4,u^5]])\cong k[[x_1,\dots x_6]]/I$, where $I=(x_1^2x_3-x_2^2, x_1^5-x_2x_3, x_3^2-x_1^3x_2,x_4,x_5,x_6)\cap (x_1,x_2,x_3,x_4^2x_6-x_5^2, x_4^5-x_5x_6, x_6^2-x_4^3x_5)$. The ring $R$ is an algebroid curve with two branches, so it is reduced and analytically unramified. To this kind of rings it is possible to associate a value semigroup $v(R)$ contained in $\mathbb{N}^2$, generalizing classical facts about the one branch case, and many properties 
at ring level can be translated and studied at semigroup level.
For details about this technique we refer to \cite{BDF} and \cite{D}.
In this case the value semigroup of $R$ is depicted with dots in Picture 1.
A reduction of the maximal ideal is given by any element of minimal value in $\mathfrak m$, e.g. $(t^3,u^3)$. If we compute $\mathfrak m(t^3,u^3)^{-1}$, we get a ring and, therefore, it coincides with $B=\mathfrak m: \mathfrak m$ which means that $R$ has minimal multiplicity. We depicted the value semigroup of $B$ in Picture 2 and we notice that $B$ is not local.

A canonical ideal $R \subseteq \omega_R \subseteq \overline R$ is characterized by its value set $K$, which can be described numerically (\cite[Theorem 4.1]{D}). We have that $v(R) \subseteq K
\subseteq \mathbb N^2$ and we depicted the points of $K\setminus v(R)$
 with circles in Picture 1; more precisely, in this case, $\omega_R$ can be chosen as the $k$-vector space generated by $(t^n,u^m)$ with $(n,m)\in K$. 
 
It is possible to check that $v(J(B))=\{(m,n) \mid m\geq 3, n\geq 3 \}$ and that $v(\m \omega_R J(B))=v(\mathfrak m)+K+v(J(B))$. Therefore, it contains only points $(m,n)$ with $m\geq 6$ and $n \geq 6$, that means  $\m \omega_R J(B)\subseteq 
(R:\overline R) \subset R$. It follows that $R$ is gAGL. 
We notice that $\omega_R^2=\omega_R^3= \mathbb N^2$ and using value sets it is possible to compute lengths (see \cite[Section 2]{D} and \cite[Section 2.1]{BDF}): we get that $\ell _R(\omega_R^2/\omega_R)=3$ and so $R$ is not 2-AGL. Moreover, $B$ has two maximal ideals and both its localizations are almost Gorenstein and not Gorenstein; indeed they are isomorphic to $k[[u^3,u^4,u^5]]$ and $k[[t^3,t^4,t^5]]$ respectively. Therefore, $B$ is almost Gorenstein being almost Gorenstein both its localizations. 

In \cite[Corollary 5.3]{CGKM} it is proved that, if $R$ has minimal multiplicity and $(B,\n)$ is local, $R$ is 2-AGL if and only if $B$ is almost Gorenstein and $R/\m \cong B/\n$.
This example shows that the notion of 2-AGL ring is not enough to characterize the almost Gorensteinnes of $B$, even if $R$ has minimal multiplicity and all the residue fields of $B$ are isomorphic to $R/\m$.
In Corollary \ref{final corollary} we are going to show that the notion of gAGL ring is a step forward in this direction. 

\begin{picture}(320,200)(-40,-60)
\put(0,0){\line(1,0){130}}
\put(0,0){\line(0,1){130}}
\put(0,0){\circle*{3}} \put(30,30){\circle*{3}} 
\multiput(60,30)(10,0){4}{\circle*{3}}
\multiput(60,60)(10,0){4}{\circle*{3}}
\multiput(60,70)(10,0){4}{\circle*{3}}
\multiput(60,80)(10,0){4}{\circle*{3}}
\multiput(60,90)(10,0){4}{\circle*{3}}
\multiput(30,60)(0,10){4}{\circle*{3}}
\multiput(30,0)(10,0){7}{\circle{3}}
\multiput(30,10)(10,0){7}{\circle{3}}
\multiput(30,40)(10,0){7}{\circle{3}}
\multiput(0,30)(0,10){7}{\circle{3}}
\multiput(10,30)(0,10){7}{\circle{3}}
\multiput(40,50)(0,10){5}{\circle{3}}
\put(10,0){\circle{3}} \put(0,10){\circle{3}} \put(30,50){\circle{3}} \put(10,10){\circle{3}}
\put(40,30){\circle{3}} \put(50,30){\circle{3}} \put(50,50){\circle{3}}
\multiput(100,30)(5,0){4}{\circle*{1.5}}
\multiput(100,0)(5,0){4}{\circle{1.5}}
\multiput(100,10)(5,0){4}{\circle{1.5}}
\multiput(100,40)(5,0){4}{\circle{1.5}}
\multiput(100,60)(5,0){4}{\circle*{1.5}}
\multiput(100,70)(5,0){4}{\circle*{1.5}}
\multiput(100,80)(5,0){4}{\circle*{1.5}}
\multiput(100,90)(5,0){4}{\circle*{1.5}}
\multiput(100,100)(5,0){4}{\circle*{1.5}}
\multiput(100,105)(5,0){4}{\circle*{1.5}}
\multiput(100,110)(5,0){4}{\circle*{1.5}}
\multiput(100,115)(5,0){4}{\circle*{1.5}}
\multiput(30,100)(0,5){4}{\circle*{1.5}}
\multiput(0,100)(0,5){4}{\circle{1.5}}
\multiput(10,100)(0,5){4}{\circle{1.5}}
\multiput(40,100)(0,5){4}{\circle{1.5}}
\multiput(60,100)(0,5){4}{\circle*{1.5}}
\multiput(70,100)(0,5){4}{\circle*{1.5}}
\multiput(80,100)(0,5){4}{\circle*{1.5}}
\multiput(90,100)(0,5){4}{\circle*{1.5}}
\put(27,-15){\small 3}
\put(57,-15){\small 6}
\put(67,-15){\small 7}
\put(78,-15){$\rightarrow$}
\put(-15,27){\small 3}
\put(-15,57){\small 6}
\put(-15,67){\small 7}
\put(-15,80){$\uparrow$}
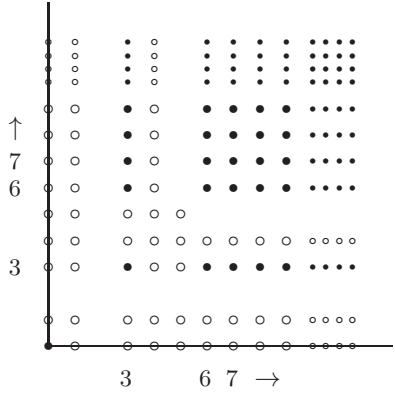
\put(-10,-40){Picture 1. $v(R)$ and $v(\omega_R)$}

\put(260,0){\line(1,0){100}}
\put(260,0){\line(0,1){100}}
\put(260,0){\circle*{3}}
\multiput(290,0)(10,0){4}{\circle*{3}}
\multiput(260,30)(0,10){4}{\circle*{3}}
\multiput(290,30)(0,10){4}{\circle*{3}}
\multiput(300,30)(0,10){4}{\circle*{3}}
\multiput(310,30)(0,10){4}{\circle*{3}}
\multiput(320,30)(0,10){4}{\circle*{3}}
\multiput(330,0)(5,0){4}{\circle*{1.5}}
\multiput(330,30)(5,0){4}{\circle*{1.5}}
\multiput(330,40)(5,0){4}{\circle*{1.5}}
\multiput(330,50)(5,0){4}{\circle*{1.5}}
\multiput(330,60)(5,0){4}{\circle*{1.5}}
\multiput(330,70)(5,0){4}{\circle*{1.5}}
\multiput(330,75)(5,0){4}{\circle*{1.5}}
\multiput(330,80)(5,0){4}{\circle*{1.5}}
\multiput(330,85)(5,0){4}{\circle*{1.5}}
\multiput(260,70)(0,5){4}{\circle*{1.5}}
\multiput(290,70)(0,5){4}{\circle*{1.5}}
\multiput(300,70)(0,5){4}{\circle*{1.5}}
\multiput(310,70)(0,5){4}{\circle*{1.5}}
\multiput(320,70)(0,5){4}{\circle*{1.5}}
\put(287,-15){\small 3}
\put(297,-15){\small 4}
\put(307,-15){\small 5}
\put(318,-15){\small 6}
\put(328,-15){$\rightarrow$}
\put(245,27){\small 3}
\put(245,37){\small 4}
\put(245,47){\small 5}
\put(245,57){\small 6}
\put(245,70){$\uparrow$}
\put(260,-40){Picture~2. $v(B)$}
\end{picture}

\end{example}

Despite the numerical semigroup case, there exist gAGL rings that are not 2-AGL even if $\m^2=\alpha \m$ and $B$ is local. We show such a ring in the following example.

\begin{example} \label{Example 1} \rm
Consider the generalized semigroup ring $R=\mathbb{R}+t^3 \mathbb{C}[[t^3,t^4,t^5]]=\mathbb{R}+\mathbb{C}t^3+t^6 \mathbb{C}[[t]] \subseteq \mathbb{C}[[t]]$, where $t$ is an indeterminate. By \cite[Theorem 5]{Se} a canonical module of $R$ included between $R$ and its integral closure is $\omega_R=\mathbb{C}+\mathbb{C}t+\mathbb{C}t^3+\mathbb{C}t^4+\mathbb{R}t^5+t^6\mathbb{C}[[t]]$ and $\omega_R^2=\mathbb{C}[[t]]$. Therefore, $\ell_R(\omega_R^2/\omega_R)=3$ and, then, $R$ is not 2-AGL. On the other hand, since $B=\mathbb{C}[[t^3,t^4,t^5]]$ and $R:\omega_R=(t^3,it^3)^2+(t^7,it^7,t^8,it^8)=t^6 \mathbb{C}[[t]]$, it is straightforward to check that $B$ is local with maximal ideal $\n=(R:\omega_R):\m=t^3 \mathbb{C}[[t]]$ and, thus, $R$ is gAGL, even if $\m^2=t^3 \m$. Clearly, in this case we have $R/\m \cong \mathbb{R} \ncong \mathbb{C} \cong B/\n$. We also notice that in this example $B$ is almost Gorenstein, since the associated numerical semigroup is almost symmetric. 
\end{example}

\section{The endomorphism algebra $\m:\m$}

As in the previous section, we assume that $(R,\m,k)$ is a one-dimensional Cohen-Macaulay local ring with $k$ infinite and that there exists a canonical module $\omega_R$ of $R$ such that $R \subseteq \omega_R \subseteq \overline{R}$. We denote by $\alpha$ a minimal reduction of $\m$ and we recall that $B$ denotes the algebra $\m:\m$.

\begin{proposition} \label{Canonical ideal of m:m}
Assume that $R$ is not a DVR and set $\omega_{B}=\alpha^{-1}(\omega_R : B)$. Then, $\omega_{B}$ is a canonical module of $B$ and $B \subseteq \omega_{B} \subseteq \overline{B}=\overline{R}$.
\end{proposition}

\begin{proof} We first show that $\omega_R:B$ is a canonical module of $B$. It is easy to see that it is a fractional ideal of $B$, since it is a fractional ideal of $R$ and given $a \in (\omega_R:B), b, b' \in B$, we have $(ab)b'=a(bb') \in \omega_R$, that is $ab \in (\omega_R:B)$. Hence, by \cite[Satz 3.3]{HK}, it is enough to show that 
\[
\ell_{B}\left( \frac{(\omega_R:B):\n}{\omega_R:B}\right)=1
\]
for every maximal ideal $\n$ of $B$. On the other hand we have
\[
\ell_{B}\left( \frac{(\omega_R:B):\n}{\omega_R:B}\right)=\ell_{B}\left( \frac{\omega_R:(\n B)}{\omega_R:B}\right)=\ell_{B}\left( \frac{\omega_R:\n}{\omega_R:B}\right).
\]
Assume by contradiction that there exists a $B$-module $L$ such that $(\omega_R:B) \subsetneq L \subsetneq (\omega_R : \n)$. Clearly, $L$ is also an $R$-module and, then, $\n \subsetneq (\omega_R:L) \subsetneq B$. Moreover, $\omega_R : L$ is also a $B$-module, indeed given $a \in (\omega_R:L)$, $b \in B$ and $\ell \in L$, we have $(ab)\ell=a(\ell b) \in \omega_R$. This yields a contradiction because $\n$ is a maximal ideal of $B$. Therefore, $\omega_R:B$ is a canonical module of $B$ and obviously also $\omega_{B}$ is. 

Moreover, by $\alpha \in \m$, it follows that $\alpha B B \subseteq \m \subseteq \omega_R$ and, then, $B \subseteq \alpha^{-1}(\omega_R : B)= \omega_{B}$. So, we only need to prove that $\omega_R :B \subseteq \alpha \overline{R}$, which, dualizing with respect to $\omega_R$, is equivalent to $\omega_R : \alpha \overline{R} \subseteq B$, i.e. $(\omega_R:\overline{R}) \subseteq \alpha B$. Since $1 \in \omega_R$, it holds $\omega_R:\overline{R}=\omega_R:\omega_R \overline{R}=(\omega_R:\omega_R):\overline{R}=R:\overline{R}$. Moreover, if $y \in (R:\overline{R})$ and $m \in \m$, it follows $y\alpha^{-1}m \in R$ because $\alpha^{-1}m \in \overline{R}$.
Therefore, $y \in (R:\alpha^{-1}\m)=\alpha(R:\m)=\alpha B$. 
\end{proof}

The third part of the next proposition is a simpler proof of \cite[Theorem 5.1]{GMP}, whereas the first and the second part are related to \cite[Theorem 3.2 and Corollary 3.5]{B} which imply that, if $R$ is almost Gorenstein, then $\alpha^{-1}\m$ is a canonical ideal of $B$ and the converse holds if $R$ is analytically unramified and residually rational.

\begin{proposition} \label{Corollary of canonical ideal of m:m}
Let $\omega_{B}=\alpha^{-1}(\omega_R : B)$. The following hold: 
\begin{enumerate}
\item If $R$ is not a DVR, then $R$ is almost Gorenstein if and only if $\omega_{B}=\alpha^{-1}\m$;
\item $R$ is almost Gorenstein if and only if $\m$ is a canonical module of $B$;
\item The ring $B$ is Gorenstein if and only if $R$ is almost Gorenstein and $\alpha \m=\m^2$.
\end{enumerate}
\end{proposition}

\begin{proof}
(1) It holds
\[\omega_{B}=\alpha^{-1}\m \ \Longleftrightarrow \ \omega_R:B=\m \ \Longleftrightarrow  B=\omega_R:\m.
\]
If $R$ is almost Gorenstein, then $\omega_R \m=\m$ implies $\omega_R:\m=\omega_R:\omega_R\m= (\omega_R:\omega_R):\m=R:\m=B$.
Conversely, we have $\omega_R \subseteq (\omega_R:\m)=B$ and, then, $\omega_R \m\subseteq \m$. \\
(2) We can assume that $R$ is not a DVR and by (1) it is enough to show that if $\m$ is a canonical ideal of $B$, then $\omega_B=\alpha^{-1}\m$.
We first note that $B[\omega_B]$ is a finitely generated $B$-algebra and it is included in $\overline{B}$, then it is also a finitely generated $B$-module. Let $\omega$ be a canonical module of $B$ such that $B\subseteq \omega \subseteq B[\omega_B]$. Then, $(\omega:B[\omega_B])=(\omega:\omega B[\omega_B])=((\omega:\omega):B[\omega_B])=(B:B[\omega_B])$. It follows that $\ell_B(B[\omega_B]/\omega)=\ell_B((\omega:\omega)/(\omega:B[\omega_B]))=\ell_B(B/(B:B[\omega_B]))$ and
$\ell_B(B[\omega_B]/B)=\ell_B(B[\omega_B]/\omega)+\ell_B(\omega/B)=\ell_B(B/(B:B[\omega_B]))+\ell_B(\omega/B)$. In particular, $\ell_B(\omega/B)$ does not depend on the choice of $\omega$. Since $\alpha^{-1}\m$ is a canonical ideal of $B$ and $B \subseteq \alpha^{-1}\m \subseteq \omega_B \subseteq B[\omega_B]$, it follows that $\ell_B(\alpha^{-1}\m/B)=\ell_B(\omega_B/B)$ and, thus, $\omega_B=\alpha^{-1}\m$. \\
(3) We assume that $R$ is not a DVR or the statement is trivially true.
If $B$ is Gorenstein, then $B=\omega_{B} =\alpha^{-1} \omega_{R}:B$. It follows $B=(\omega_R :\alpha B) \supseteq (\omega_R : \m) \supseteq (\omega_R :R)=\omega_R$. Thus, $\omega_R \m \subseteq \m$ and, then, $R$ is almost Gorenstein. Moreover, (1) implies $B=\omega_{B}= \alpha^{-1}\m$, which is equivalent to $\alpha\m=\m^2$.
Conversely, we immediately get that $\omega_{B}=\alpha^{-1}\m=B$ and, then, $B$ is Gorenstein.
\end{proof}

In the rest of the paper we fix $\omega_B = \alpha^{-1}(\omega_R:B)$.

\begin{lemma} \label{m:n contains omega:n}
If $R$ is not a DVR, the following properties hold: 
\begin{enumerate}
\item A reduction of $(\omega_{B}: \alpha^{-1}\m)_{\n}$ is $1$ for every $\n \in \Max (B)$;
\item $\alpha^{-1}\m$ is an almost canonical ideal of $B$ if and only if $(\omega_R:J(B)) \subseteq (\m:J(B))$.
\end{enumerate}
\end{lemma}

\begin{proof}
(1) Since $\alpha \in \m$, it follows that $(\omega_{B}:\alpha^{-1}\m) \subseteq \omega_{B} \subseteq \overline{B}$.
Moreover, the $B$-module $\omega_{B}$ is finitely generated, since it is a canonical module, so also $\omega_{B}:\alpha^{-1}\m$ is finitely generated. Finally we get $\omega_{B}:\alpha^{-1}\m=\alpha^{-1}(\omega_R:B):\alpha^{-1}\m=\omega_R:\m B=\omega_R:\m$ and, thus, $1\in \omega_{B}:\alpha^{-1}\m$. Now it is enough to apply Lemma \ref{z=1}.  \\
(2) We note that $\m \subseteq (\omega_R:B)$ and, then, $\alpha^{-1}\m \subseteq \omega_B$. Therefore, by (1) and by Corollary \ref{Characterization almost canonical ideal}, the ideal $\alpha^{-1}\m$ is almost canonical if and only if $ (\omega_{B}:J(B)) \subseteq (\alpha^{-1}\m:J(B))$, since the other inclusion is always true. Moreover, $\omega_{B} : J(B)= \alpha^{-1} (\omega_R:B):J(B)=\alpha^{-1}(\omega_R: J(B)B)=\alpha^{-1}(\omega_R:J(B))$ and the thesis follows. 
\end{proof}

We are ready to prove the main result of this section.

\begin{theorem} \label{main}
The ring $R$ is {\rm gAGL} if and only if $\m$ is an almost canonical ideal of $B$.
\end{theorem}

\begin{proof}  
We recall that $\m$ is an almost canonical ideal of $B$ if and only if $\alpha^{-1}\m$ is, see Remark \ref{First remark}.3.\\
Assume first that $R$ is gAGL. By Proposition \ref{Corollary of canonical ideal of m:m} we can assume that $R$ is not almost Gorenstein, otherwise $\alpha^{-1}\m$ is a canonical ideal of $B$. By the previous lemma and Corollary \ref{Characterization almost canonical ideal}(3) we only need to prove that $J(B) \omega_B \subseteq \alpha^{-1}\m$, i.e. $\omega_B \subseteq (\alpha^{-1}\m:J(B))$. Moreover, since $R$ is gAGL we have
\begin{gather*}
(\omega_R:B)(B:\omega_R) \subseteq \m \Longleftrightarrow  (\omega_R:B) \subseteq (\m:(B:\omega_R)) \Longleftrightarrow 
(\omega_R:B) \subseteq (\m:((R:\m):\omega_R)) \Longleftrightarrow \\ \alpha^{-1}(\omega_R:B) \subseteq (\alpha^{-1}\m:((R:\omega_R):\m))  \Longrightarrow \omega_B \subseteq (\alpha^{-1}\m:J(B)) \Longrightarrow \alpha^{-1}\m \text{ is almost canonical}.
\end{gather*}

Since $(\omega_R:B)(B:\omega_R) \subseteq \omega_R:\omega_R=R$, it is enough to show that there are no units of $R$ in $(\omega_R:B)(B:\omega_R)$. If there exist $x \in (\omega_R:B)$ and $y \in (B:\omega_R)$ such that $xyu=1$ for some unit $u$ of $R$, then $xB \subseteq \omega_R$, $y\omega_R \subseteq B$ and $y^{-1}=xu$ imply that $xB \subseteq \omega_R \subseteq xuB$. Since $u$ is also a unit of $B$, it follows $uB=B$ and $\omega_R=xB$. In particular, we get $\omega_R:B=xB:B=xB$ and, then, $\omega_B=\alpha^{-1}(\omega_R:B)=\alpha^{-1}xB$. Since $\alpha^{-1} x$ is invertible in $Q(R)$, $B$ is Gorenstein. Hence, Proposition \ref{Corollary of canonical ideal of m:m} implies that $R$ is almost Gorenstein, which is a contradiction. \\
Conversely, we can assume that $R$ is not Gorenstein.
We prove first that $\m^2 \subseteq (R:\omega_R) \subseteq \m$. The second inclusion follows by  $1 \in \omega_R$ and $R \subsetneq \omega_R$. As for the first one, it is enough to consider the elements $m_1 m_2$ for $m_1, m_2 \in \m$. If $\lambda \in \omega_R$, then $m_2 \lambda \in (\omega_R : J(B)) \subseteq (\m:J(B))$ by Lemma \ref{m:n contains omega:n}. Moreover, $m_1 \in \m \subseteq J(B)$ and, then, $m_1 m_2 \lambda \in \m \subseteq R$, which implies $m_1 m_2 \in (R: \omega_R)$. This means that $\m^2 \subseteq (R:\omega_R)$ and, then, $R:\omega_R=(x_1, \dots, x_r)^2+ (x_{r+1}, \dots, x_n)$ for some integer $r$, where $x_1, \dots x_n$  is a minimal system of generators of $\m$. 

By Lemma \ref{Def gAGL} we only need to show that there are no $x \in J(B)$ such that  $x_i=x x_j$ for some $1 \leq i,j \leq r$. 
Assume by contradiction that this is not true. Since $x_i \notin (R:\omega_R)$, there exists $\lambda \in \omega_R$ such that $x_i \lambda \notin R$. Moreover, for every $\gamma \in J(B) \subseteq B$ we have $\gamma x_j \in \m$ and, then, $\gamma x_j \lambda \in \omega_R$ implies $x_j \lambda \in (\omega_R : J(B)) \subseteq (\m :J(B))$ again by Lemma \ref{m:n contains omega:n}. Consequently, $x_i \lambda = x x_j \lambda \in \m \subseteq R$ yields a contradiction. 
\end{proof}

\begin{remark} \rm \label{independence}
Since being almost canonical does not depend on the choice of the canonical module, Theorem \ref{main} implies that the gAGL property is independent of the chosen canonical module. 
\end{remark}


\begin{corollary} \label{final corollary}
If $B$ is almost Gorenstein, then $R$ is {\rm gAGL}. The converse holds if $R$ has minimal multiplicity.
\end{corollary} 

\begin{proof} We note that $B \subseteq \alpha^{-1} \m$ are fractional ideals of $B$ and the inclusion $\alpha^{-1}\m \subseteq \alpha^{-1}(\omega_R:B)=\omega_B$ implies that $B \subseteq (\omega_B:(\alpha^{-1} \m)) \subseteq (\omega_B:B)=\omega_B \subseteq \overline{B}$. 
Therefore, if $B$ is almost Gorenstein, Corollary \ref{almost canonical and almost Gorenstein} and Remark \ref{containment between fractional ideals} imply that $\alpha^{-1} \m$ is an almost canonical ideal of $B$ and, thus, the first statement follows by the previous theorem and Remark \ref{First remark}.3.
As for the last part it is enough to recall that $B$ is almost Gorenstein if and only if $B=\alpha^{-1}\m$ is an almost canonical ideal.
\end{proof}

By Proposition \ref{2-AGL and gAGL}, if $R$ has minimal multiplicity, $(B,\n)$ is local and $B/\n \cong R/\m$, the previous corollary says that $B$ is almost Gorenstein if and only if $R$ is 2-AGL, that is \cite[Corollary 5.3]{CGKM}. Therefore, it gives another proof of this fact. On the other hand it generalizes \cite[Corollary 5.3]{CGKM} even if $B$ is local or $B/\n \cong R/\m$ for every $\n \in \Max(B)$. For instance in both Example \ref{Example 2} and Example \ref{Example 1} the ring $B$ is almost Gorenstein local and $R$ is gAGL but not 2-AGL.

\begin{example} \rm
{\bf 1.} If $R$ has not minimal multiplicity, the converse of Corollary \ref{final corollary} does not hold. For example the semigroup ring $k[[t^6,t^7,t^{15},t^{17}]]$ is 2-AGL and then g-AGL, but $B=k[[t^6,t^7,t^8,t^{11}]]$ is not almost Gorenstein. \\
{\bf 2.} By Proposition \ref{Corollary of canonical ideal of m:m} the Gorensteinness of $B$ implies that $R$ has minimal multiplicity. This is no longer true if $B$ is only almost Gorenstein. For instance the gAGL ring $R=k[[t^5,t^6,t^9,t^{13}]]$ has multiplicity $5$ and embedding dimension $4$ and $B=k[[t^4,t^5,t^6,t^7]]$ is almost Gorenstein. 
\end{example}

\end{document}